\documentclass[reqno,10pt]{amsart}
\usepackage{amsfonts}
\usepackage[english]{babel}
\usepackage{amsthm,hyphenat}
\usepackage{amsmath}
\usepackage{amssymb}
\usepackage[latin1]{inputenc}
\usepackage[all]{xy}
\usepackage{dsfont}
\usepackage[bookmarksnumbered]{hyperref}
\usepackage{xcolor}
\usepackage{graphics}
\usepackage{enumerate}
\usepackage{mathrsfs} 
\usepackage[normalem]{ulem}
\def\bb#1\eb{{\color{blue}
{#1}}} %
\def\br#1\er{{\color{red}
{#1}}} %
\def\bv#1\ev{{\color{gray}
		{#1}}} %
		\def\bm#1\em{{\color{magenta}
		{#1}}} %
\newcommand{\R}{\mathds R}

\newcommand{\h}{h}

\title[Anisotropic conformal invariance of lightlike geodesics]{Anisotropic conformal invariance of lightlike geodesics in  pseudo-Finsler manifolds}

\author[M. A. Javaloyes]{Miguel Angel Javaloyes}
\address{Departamento de Matem\'aticas, \hfill\break\indent
Universidad de Murcia, \hfill\break\indent
Campus de Espinardo,\hfill\break\indent
30100 Espinardo, Murcia, Spain}
\email{majava@um.es}

\author[B. L. Soares]{Bruno Learth Soares}
\address{}
\email{bsoares@ime.usp.br}

\date{}
 \thanks{ This work is a result of the activity developed within the framework of the Programme in
 	Support of Excellence Groups of the Regi\'on de Murcia, Spain, by Fundaci\'on S\'eneca, Science and Technology Agency of the Regi\'on de Murcia. The first author was partially supported by MICINN/FEDER project with reference PGC2018-097046-B-I00  and Fundaci\'on S\'eneca project with reference 19901/GERM/15. }

\thanks{2000 {\it Mathematics Subject Classification:} Primary 53C22, 53C50, 53C60, 58B20\\
\textbf{Key words:} Finsler, Lightlike geodesics, Index form, focal points.}

\begin{document}
\newtheorem{thm}{Theorem}[section]
\newtheorem{prop}[thm]{Proposition}
\newtheorem{lemma}[thm]{Lemma}
\newtheorem{cor}[thm]{Corollary}
\theoremstyle{definition}
\newtheorem{defi}[thm]{Definition}
\newtheorem{notation}[thm]{Notation}
\newtheorem{exe}[thm]{Example}
\newtheorem{conj}[thm]{Conjecture}
\newtheorem{prob}[thm]{Problem}
\newtheorem{rem}[thm]{Remark}

\begin{abstract}
In this paper, we prove that lightlike geodesics  of a pseudo-Finsler manifold  and its focal points are preserved up to reparametrization by  anisotropic conformal changes, using  the Chern connection and the anisotropic calculus  developed in \cite{J16,J19}  and the fact that geodesics are critical points of the energy functional and Jacobi fields, the kernel of its index form. This result has applications to the study of Finsler spacetimes. 
\end{abstract}

\maketitle
\section{Introduction}
\ Conformal Geometry has been studied in Riemannian Geometry from its very beginning, since it is related with the angle-preserving maps and it has applications in navigation chart making and many other things.  In the Lorentzian context, the relevance of conformal geometry increases as it preserves causality and it has been extensively used in studying the boundary of a spacetime \cite{Frau04}. In the realm of Finsler metrics, one can consider a much more general concept of conformal geometry making the conformal factor dependent on the direction, namely, a function on the tangent bundle, which we call {\it anisotropic conformal geometry}. \ The concept of anisotropic conformal geometry seems to be too general in the context of Finsler Geometry, since every two Finsler metrics are anisotropically conformal. But surprisingly, it gains relevance when one considers  the more general class of  pseudo-Finsler metrics. It turns out that two pseudo-Finsler metrics are anisotropically conformally equivalent if and only if they have the same lightcone (see 
Def.~\ref{def:anisot} and Th. \ref{t_anisotropic}).  
The first natural question to study is whether it happens as in the classical Lorentzian case, where lightlike pregeodesics are preserved by conformal changes (see for example \cite[Th. 2.36]{MinSan08}). The answer is positive (Prop. \ref{georepara}), which is proved by using that geodesics are critical points of the energy functional, and computing the Euler-Lagrange equations corresponding to the first variation of the energy of an anisotropic conformal pseudo-Finsler manifold $(M,\lambda L)$, in terms of the Chern connection of the pseudo-Finsler manifold $(M,L)$. 

 As a further result,  in the last section, we also prove that $P$-focal points are preserved with multiplicity by anisotropic conformal changes, Th. \ref{finalresult}. This is done using that $P$-Jacobi fields are the kernel of the second variation of the energy functional and making all the computations with the Chern connection of $(M,L)$. The essence of the proof lies in the result of  Lemma  \ref{PQJacobi},
where we prove that given a $(P,Q)$-Jacobi field $\tilde J$ of a geodesic $\tilde \gamma$  of $(M,L)$ with $P$ and $Q$ submanifolds of $M$ orthogonal to $\tilde\gamma$ at its endpoints, one can construct a $(P,Q)$-Jacobi field of its reparametrization $\gamma$ as a geodesic of $(M, \lambda L)$ by reparametrizing $\tilde J$ and adding a multiple of the velocity vector $\dot\gamma$. With this procedure, we avoid at all moment to work with the Chern connection of $(M,\lambda L)$. All the computations are made using the Chern connection and the anisotropic calculus described in \cite{J16,J19}, which makes them more available to classical Riemannian Geometers. 

 One of the most fashionable applications of our results is the study of light rays in Finsler spacetimes. Let us observe that there are some examples of Finsler spacetimes with a $2$-homogeneous Finsler metric which is not smooth on the lightcone as for example Bogoslovsky spacetimes \cite{Bogos77} (see also \cite{FPP18} and \cite[\S 6.1.1]{BJS20}), Randers spacetimes \cite{Sta12}, Kostelecky models \cite{Kos11}, or those given by bimetrics (see \cite{PW11}). In some cases, the lightcone is still a cone structure in the sense of \cite{JS18} which admits a smooth Lorentz-Finsler metric. The results on this work claim that lightlike geodesics are well-defined, namely, do not depend on the choice of Lorentz-Finsler metric for the cone structure and they  coincide with the cone geodesics of \cite[\S 6]{JS18}. But more importantly, the focal points, which are of great importance in General Relativity (see for example the Finsler version of Penrose's Singularity Theorem \cite{AJ16})
are also preserved. Moreover, there are some difficulties to obtain the Einstein's equations on the lightcone \cite{HPV19}. Our results show that lightlike geodesics and its focalization can be controled by knowing only the boundary of timelike vectors. The only thing that it is not possible to control is the parametrization of lightlike geodesics, which is important for example in the Singularity Theorems.  

\section{Preliminaries on pseudo-Finsler metrics} 
Let $M$ be an  $n$-dimensional manifold and denote by $\pi:TM\rightarrow M$ the natural projection
of the tangent bundle $TM$ into $M$. Let $A\subset TM\setminus  {\bf 0}$ be an open subset of $TM$ which is conic, that is, such that $\pi(A)=M$ and $\lambda v\in A$, for every $v\in A$ and $\lambda>0$. We say that a smooth function
$L:A\rightarrow \R$ is a (conic, two homogeneous) {\it pseudo-Finsler metric} if
\begin{enumerate}[(i)]
\item $L$ is positive-homogeneous of degree $2$, that is, $L(\lambda v)=\lambda^2 L(v)$ for every $v\in A$ and $\lambda>0$,
\item for every $v\in A$, {\it the fundamental tensor $g_v$ of $L$ at 
$v$} defined by
\[g_v(u,w):=\frac 12 \frac{\partial^2}{\partial t\partial s} L(v+tu+sw)|_{t=s=0},\]
for any $u,w\in T_{\pi(v)}M$, is nondegenerate.
\end{enumerate}
 Clearly,  the fundamental tensor is bilinear and symmetric. We will  refer to the pair $(M,L)$, being $M$ a manifold and $L$ a pseudo-Finsler metric on $M$, as a pseudo-Finsler manifold. 
 
 Basic properties of homogeneous functions imply  the following. 
\begin{prop}\label{fundamentalprop}
Given a pseudo-Finsler metric $L$ and $v\in A$, the fundamental tensor 
$g_v$ is positive homogeneous of degree 0, that is, $g_{\lambda v}=g_v$ for $\lambda>0$. Moreover $g_v(v,v)=L(v)$ and
$ g_v(v,w)= \frac 12\frac{\partial}{\partial z}L\left(v+z w\right)|_{z=0}=\frac 12 dL_v(w)$.
\end{prop}

\subsection{Cartan tensor} In Finsler geometry, unlike the Riemannian setting, we need to consider the third  vertical derivatives of the metric in order to define a connection. This information is contained in the {\it Cartan tensor}, which is defined as the trilinear symmetric form
\begin{equation}\label{cartantensor}
C_v(w_1,w_2,w_3)=\frac 14\left. \frac{\partial^3}{\partial s_3\partial s_2\partial s_1}L\left(v+\sum_{i=1}^3 s_iw_i\right)\right|_{s_1=s_2=s_3=0},
\end{equation}
for $v\in A$ and $w_1,w_2,w_3\in T_{\pi(v)} M$.

 The following is another simple consequence of basic properties of 
homogeneous functions: 
\begin{prop}\label{Cartanprop}
 The Cartan tensor is homogeneous of degree $-1$, that is, $C_{\lambda v}=\frac{1}{\lambda}C_v$ for any $v\in A$ and $\lambda>0$. Moreover,  $C_v(v,w_1,w_2)=C_v(w_1,v,w_2)=C_v(w_1,w_2,v)= 0$ for every $v\in A$ and  $w_1,w_2\in T_{\pi(v)}M$.
\end{prop}

\newcommand{\C}{\mathcal{C}}
\subsection{Pseudo-Finsler metrics with the same lightlike cone} Let us show that fixing  the lightlike cone of a pseudo-Finsler metric is the same  as fixing the anisotropic conformal class. 
First, we will see some properties of the {\em lightcone} of a pseudo-Finsler metric $L:A\rightarrow \R$, defined as $\C=\{v\in A: L(v)=0\}$.
\begin{lemma}\label{lem:previous}
A  pseudo-Finsler metric $L:A\rightarrow \R$ in a manifold $M$ has $0$ as a regular value and the lightcone $\C$ of $L$   (if not empty)  is  a smooth  hypersurface of $TM$, transversal to all the tangent spaces $T_pM$, with $p\in M$.
	\end{lemma}
\begin{proof}
	Observe that $0$ is a regular value of $L$, because by Prop. \ref{fundamentalprop}, $d_vL(w)= 2g_v(v,w)$,  and since $g_v$ is nondegenerate, $d_vL$ cannot be identically zero  (recall that $0\notin A$), and then for a fixed $\tilde v\in A$, there exists a vector $\tilde w$ such that $g_{\tilde v}(\tilde v,\tilde w)\not=0$.  Moreover, it follows from the implicit function theorem that $L^{-1}(0)=\C$   (if not empty)  is a smooth hypersurface and that the vector  $\tilde w$ is transversal to $\C$  at $\tilde v$.  Hence,  $\C$ is transversal to every $T_pM$, $p\in M$.
	\end{proof}
\begin{defi}\label{def:anisot}
	We will say that two pseudo-Finsler metrics defined in the same open conic subset $L_1,L_2:A\rightarrow \R$ are {\em anisotropically equivalent} if there exists a smooth function $\mu:A\rightarrow\R$ without zeroes such that $L_2=\mu L_1$.
	\end{defi}
 Observe that if $L_1$ and $L_2$ do not vanish away from the zero section, then they are always anisotropically equivalent, as it happens for example with classical Finsler metrics. The interest of this concept comes into play when there is a non-empty lightcone. 

The next theorem is a generalization of \cite[Th. 3.11]{JS18}, written in the context of Lorentz-Finsler metrics, to pseudo-Finsler metrics of arbitrary index.
\begin{thm}\label{t_anisotropic}
	Two pseudo-Finsler metrics $L_1,L_2: A\rightarrow \R$ are anisotropically equivalent if and only if their lightcones coincide.  
	Moreover,  in such a case,  the factor of anisotropy $\mu=L_2/L_1$ on $A$ can be computed as  
	\begin{equation}\label{e_mu}
	\mu(v)=\frac{g^2_v(v,w)}{g^1_v(v,w)},
	\end{equation}
	where $g^1$ and $g^2$ are the fundamental tensors of $L_1$ and $L_2$, respectively,  and  $w$  is any vector in $T_{\pi(v)}M$ such that $g^1_v(v,w)\neq 0$ (and, thus, $g^2_v(v,w)\neq 0$). 
\end{thm}
\begin{proof}
	Observe that if $L_1$ and $L_2$ are anisotropically equivalent, then it follows straightforwardly that they have the same lightcone. For the converse, consider a connected component $\tilde A$ of $A\setminus \{\C\cup \bf 0\}$. Then $L_1$ and $L_2$ do not change sign in $\tilde A$. Assume that both of them are positive on $\tilde A$ and apply \cite[Lemma 3.10]{JS18} taking into account Lemma \ref{lem:previous},  to conclude that $L_2=\mu L_1$, for a smooth $\mu>0$. If $L_1$ or $L_2$ are negative on $\tilde A$, apply \cite[Lemma 3.10]{JS18} to $\tilde L_1=-L_1$ and/or $\tilde L_2=-L_2$.  Observe also that a pseudo-Finsler metric always changes the sign in the lightcone, since for $v\in \C$ and $w$ transversal to $\C$ in $v$, one has that $\frac{d}{dt}L(v+tw)=2g_{v+tw}(v+tw,w)\not=0$ in a neighborhood of $t=0$. This implies that $L_1$ and $L_2$ have always either the same sign or opposite sign in a connected component of $A$ and then one can apply  \cite[Lemma 3.10]{JS18} to the whole connected component.  The last part follows the same lines as in the proof of \cite[Th. 3.11]{JS18}.
	\end{proof}

\subsection{Chern connection, covariant derivative  and curvature tensor}\label{chernsection}

Assume that $(M,L)$ is a pseudo-Finsler manifold with domain $A\subset TM$, and denote by $\mathfrak X(M)$ the module of smooth vector fields on $M$.    Let us introduce the approach to connections in Finsler Geometry collected in \cite{J16,J19}.
	An  {\em anisotropic (linear) connection} is  a map
\[\nabla: A\times \mathfrak{X}(M)\times\mathfrak{X}(M)\rightarrow TM,\quad\quad (v,X,Y)\mapsto\nabla^v_XY:=\nabla(v, X,Y)\in T_{\pi(v)}M,\]
such that
\begin{enumerate}[(i)]
	\item $\nabla^v_X(Y+Z)=\nabla^v_XY+\nabla^v_XZ$, for any $X,Y,Z\in  {\mathfrak X}(M)$,
	\item $\nabla^v_X(fY)=X(f) Y_{\pi(v)}+f(\pi(v)) \nabla^v_XY $ for any $f\in {\mathcal F}(M)$, $X,Y\in  {\mathfrak X}(M)$,
	\item for any $X,Y\in  {\mathfrak X}(M)$,  the map $A\ni v\rightarrow \nabla^v_XY$ is smooth,
	\item $\nabla^v_{fX+hY}Z=f(\pi(v))\nabla^v_XZ+h(\pi(v)) \nabla^v_YZ$, for any $f,h\in {\mathcal F}(M)$, $X,Y,Z\in  {\mathfrak X}(M)$.
\end{enumerate}
Here, ${\mathcal F}(M)$ denotes the space of real smooth functions on $M$.  We will say that $V\in\mathfrak X(\Omega)$ is {\it $A$-admissible} if $V_p\in A$ for every $p\in\Omega$. From now on,  given an $A$-admissible vector field $V$ in $\Omega$, we will construct classical tensors $g_V$ and $C_V$ on $\Omega$ using the fundamental and the Cartan tensors, namely, for every $p\in M$, these tensors are given by $g_{V_p}$ and $C_{V_p}$, respectively. In particular, $g_V$ is usually called {\it the osculating metric}  with  respect to the vector field $V$. We will also define an affine connection $\nabla^V$ as $(\nabla^{V}_XY)_p:=\nabla^{V_p}_XY$, for every $p\in \Omega$  and any $X,Y\in  {\mathfrak X}(\Omega)$.

 Then, the {\em Chern connection} is the unique anisotropic connection $\nabla$ such that for every $A$-admissible vector field $V$ defined in $\Omega\subset M$, the associated affine connection $\nabla^V$ satisfies  the following two properties: 
\begin{itemize}
 \item[(i)] $\nabla^V_XY-\nabla^V_YX=[X,Y]$ for all $X,Y\in\mathfrak X(\Omega)$ 
({\it torsion freeness});
 \item[(ii)] $X(g_V(Y,Z))=g_V(\nabla^V_XY,Z)+g_V(Y,\nabla^V_XZ)+2C_V(\nabla^V_XV,Y,Z)$ for all $X,Y,Z\in \mathfrak X(\Omega)$ ({\it almost $g$-compatibility}).
\end{itemize}
(see \cite{Mat80}, \cite[Eq. (7.20) and(7.21)]{Sh01} and \cite{J13}). It is easy to see that $\nabla$ is positive homogeneous of degree $0$ at every $v\in A$, i.e., $\nabla^v=\nabla^{\lambda v}$ for all positive $\lambda$. 

 Given a smooth function $f:A\rightarrow \R$ and a vector field $X\in \mathfrak{X}(M)$, there are two possible derivatives, the so-called vertical derivative $\partial^\nu f(X)$, which is defined as
\[\partial^\nu f_v(X)=\left.\frac{d}{dt}\right|_{t=0}f(v+tX_{\pi(v)}),\]
and the derivative induced by the Chern connection $\nabla$, which gives a function $\nabla_X f$. To define this function at $v\in A$, one can use an $A$-admissible extension $V$ of $v$ in some open subset $\Omega$, then
\begin{equation}\label{derivf}
\nabla_X f(v)=X_{\pi(v)}(f(V))-\partial^\nu f_v(\nabla^V_XV).
\end{equation}
It is not difficult to see that it is well-defined as the result does not depend on the extension $V$ (see \cite[Lemma.9]{J16}). Moreover, we can define two different  gradients of $f$ as follows. The vertical gradient at $v\in A$, $(\nabla^\nu f)_v$,  is determined by 
\begin{equation}\label{nablanu}\partial^\nu f_v(X)=g_v((\nabla^\nu f)_v,X)
\end{equation}
for every $X\in {\mathfrak X}(M)$. While the horizontal gradient at $v\in A$, $(\nabla^h f)_v$, is given by
\begin{equation}\label{nablah}
\nabla_X f (v)=g_v((\nabla^h f)_v,X),
\end{equation}
for every $X\in {\mathfrak X}(M)$. Both, $\nabla^\nu f$ and $\nabla^h f$ can be thought as anisotropic vector fields (see \cite{J16} and \cite{J19}).

 We can also define the {\em curvature tensor associated with an anisotropic connection}, which can be computed in $v\in A$ with the help of an $A$-admissible extension $V$ of $v$ in an open subset $\Omega\subset M$ as
\begin{equation}\label{RvRV}
R_v(X,Y)Z=(R^V(X,Y)Z-P_{\bv V \ev}(Y,Z,\nabla^V_XV)+P_{\bv V\ev}(X,Z,\nabla^V_YV))_{\pi(v)},
\end{equation}
where $V,X,Y,Z\in {\mathfrak{X}}(\Omega)$, $R^V$ is the curvature tensor of the affine connection $\nabla^V$ and
\[P_v(X,Y,Z)=\frac{d}{dt} \big(\left.\nabla^{v+tZ(\pi(v))}_XY\big)\right|_{t=0},\]
 (see \cite[Prop. 2.5]{J19}). 

  Now we suppose that $\Omega$ is a chart domain with coordinate system 
\[
 x=(x^1,\dots,x^n):\Omega\to x(\Omega)\subset \R^n.
\]
The {\it Christoffel symbols} of $\nabla$ with respect to the chart $(\Omega,x)$ are the smooth functions $\Gamma^k_{ij}:T\Omega\cap A\to\R$ such that
\[
 \nabla^v_{\frac\partial{\partial x^i}}\left(\frac\partial{\partial x^j}\right)=\Gamma^k_{ij}(v)\left.\frac\partial{\partial x^k}\right|_{\pi(v)};\quad i,j\in\{1,\dots,n\},
\]
 using Einstein convention for summation.
 
 \subsection{Geodesics}
 Given a smooth curve $\gamma:[a,b]\to M$, we denote by $\mathfrak X(\gamma)$ the $C^\infty([a,b])$-module of vector fields along $\gamma$. We say that $U\in\mathfrak X(\gamma)$ is $A$-admissible if $U(t)\in A$ for all $t\in [a,b]$. In particular,  we say that $\gamma$ itself is $A$-admissible if $\dot\gamma(t)\in A$ for every $t\in [a,b]$. For every $A$-admissible vector field $U\in\mathfrak X(\gamma)$, the Chern connection induces a covariant derivative $D^U_\gamma:\mathfrak X(\gamma)\to\mathfrak X(\gamma)$ along $\gamma$, given locally,  when $\gamma$ is contained in the chart domain $\Omega$,  by
\begin{equation}\label{connection}
 D^U_\gamma X=\left(\dot X^k+X^i \dot\gamma^j (\Gamma^k_{ij}\circ U)\right)\left.\frac\partial{\partial x^k}\right|_{\gamma},
\end{equation}
where $X=X^i\left.\frac\partial{\partial x^i}\right|_\gamma$, $\dot\gamma= \dot\gamma^i\left.\frac\partial{\partial x^i}\right|_{\gamma}$ (see, again, \cite[Prop. 2.6]{J13}). The induced covariant derivative is also almost $g$-compatible.

 \begin{defi}
 A smooth $A$-admissible curve $\gamma$ of a pseudo-Finsler manifold $(M,L)$ is called a {\it geodesic} if $D^{\dot\gamma}_\gamma \dot\gamma=0$.
\end{defi} 

 \begin{rem}\label{lightlike}
 If $\gamma:[a,b]\to M$ is a geodesic of $(M,L)$, then the function 
$L\circ\dot\gamma:[a,b]\to\R$ is constant. Indeed, as we will see below in \eqref{basic},   if $\phi=L\circ\dot\gamma$, then $\dot\phi=2g_{\dot\gamma}(D^{\dot\gamma}_\gamma\dot\gamma,\dot\gamma)=0$.  

\end{rem}

\section{Variations of the energy of a conformal metric}\label{section:light}
In this section, we examine the effects of conformal transformations on 
lightlike curves, that is, curves $\gamma$ such that $L\circ \dot\gamma = 0$. We prove that 
some key geometric properties of these curves (such as being a geodesic, 
and having conjugate or focal points) are preserved up to reparametrization by such 
transformations  (see also Remark \ref{lightlike}).
\subsection{First variation of the energy}
  Given a pseudo-Finsler manifold $(M,L)$,  with $L:A\subset TM\setminus {\bf 0}\rightarrow \R$,  we shall denote by $C_L(M,[a,b])$ the space of $A$-admissible  smooth curves in $M$ defined on the closed interval $[a,b]$.  Let $\lambda : A\rightarrow (0,+\infty)$ be an arbitrary positive smooth function  homogeneous of degree zero and assume that $\lambda L:A\rightarrow \R$ is also a pseudo-Finsler metric (and therefore its fundamental tensor is non-degenerate). 
We want to consider variations of the energy functional of $\lambda L$.  Given a  smooth curve $\gamma:[a,b]\rightarrow M$, let us denote
\begin{equation*}
E_{\lambda} : \gamma\in C_L(M,[a,b])\mapsto E_{\lambda}(\gamma)=\frac 12\int_a^b \lambda(\dot\gamma(t))L(\dot\gamma(t))dt,
\end{equation*}
 (from now on, we will omit to write the integration parameter). 

Throughout this section we will always use the Chern covariant derivative $D_\gamma$ along a smooth curve $\gamma$ associated with $L$. Indeed, in the following we will try to express the first and second variations of the energy $E_\lambda$ in terms of $D_\gamma$ rather than using the covariant derivative associated with $\lambda L$. Moreover, $\mathscr{L}_L$ denotes the Legendre transform of $L$, namely, the map $\mathscr{L}_L: A\rightarrow TM^*$, where $\mathscr{L}_L(v)$ is defined as the one-form given by $\mathscr{L}_L(v)(w)=g_v(v,w)$ for every $w\in T_{\pi(v)}M$. 


Let $\gamma\in C_L(M,[a,b])$, and consider
a  smooth variation $\Lambda:[a,b]\times (-\varepsilon,\varepsilon)\rightarrow M,\ (t,s)\mapsto\Lambda(t,s)$ of $\gamma$. 
  Given $s_o\in (-\varepsilon,\varepsilon)$ and $t_o\in [a,b]$, we will denote by $\gamma_{s_o}:[a,b]\rightarrow M$ the curve defined as $\gamma_{s_o}(t)=\Lambda(t,s_o)$ for every $t\in [a,b]$ and by
$\beta_{t_o}:(-\varepsilon,\varepsilon)\rightarrow M$ the curve defined as $\beta_{t_o}(s)=\Lambda(t_o,s)$ for every $s\in (-\varepsilon,\varepsilon)$,  which are the longitudinal and the transversal curves of the variation,  respectively. Moreover,  we will use the notation $\frac{\partial\Lambda}{\partial t}(t,s)=\dot\gamma_s(t)$ and 
$\frac{\partial\Lambda}{\partial s}(t,s)=\dot\beta_t(s)$, and we will denote by $W$ the variational vector field of $\Lambda$ along $\gamma$, namely, $W(t)=\frac{\partial\Lambda}{\partial s}(t,0)$ for every $t\in[a,b]$.  We will say that the variation is $A$-admissible if $\gamma_s\in C_L(M,[a,b])$ for all $s\in (-\varepsilon,\varepsilon)$. 

Notice that when we have a variation of curves (or more generally a two parameters map), the fact that the Chern connection is  torsionfree  implies the following property:
\begin{equation}\label{commut}
D_{\gamma_s}^V{\dot\beta_t}=D_{\beta_t}^V{\dot\gamma_s},
\end{equation} 
(see \cite[Prop. 3.2]{J13}).

\begin{prop}\label{firstvar2}
Assume that $\gamma:[a,b]\rightarrow M$ is an $A$-admissible 
smooth lightlike curve having an $A$-admissible piecewise smooth variation 
$\Lambda:[a,b]\times (-\varepsilon,\varepsilon)\rightarrow M$. 
Then 
\begin{multline}\label{firstvariation2}
E_{\lambda}'(0)=\frac{d}{ds}E_{\lambda}(\gamma_s)\left|_{s=0}\right. =
  \int_a^bg_{\dot\gamma}\Big(W, 
- D_\gamma^{\dot\gamma}\big( \lambda(\dot\gamma)\dot\gamma\big)\Big)dt
 + \left[\lambda(\dot\gamma)\mathscr{L}_L(\dot\gamma)(W)\right]_a^b. 
\end{multline}
\end{prop}
\begin{proof}
	Observe that for any $s\in (-\varepsilon,\varepsilon)$, we have
	\begin{multline}\label{conffirstvar}
	\frac{d}{ds}E_{\lambda}(\dot\gamma_s)
	= \frac 12 \int_a^b \frac{d}{ds}\left(\lambda(\dot\gamma_s)
	g_{\dot\gamma_s}\left(\dot\gamma_s,\dot\gamma_s\right)\right)dt \\
	= \frac 12 \int_a^b\left( \frac{d}{ds}\lambda(\dot\gamma_s)\right)
	g_{\dot\gamma_s}\left(\dot\gamma_s,\dot\gamma_s\right)dt +
	\frac 12 \int_a^b \lambda(\dot\gamma_s)
	\frac{d}{ds}g_{\dot\gamma_s}\left(\dot\gamma_s,\dot\gamma_s\right)dt.
	\end{multline}
	
	Furthermore, almost $g$-compatibility of the Chern 
	connection,  Prop. \ref{Cartanprop} and  the identity \eqref{commut} 
	 imply
	\begin{align}
	\nonumber\frac 12 \frac{d}{ds} g_{\dot\gamma_s}(\dot\gamma_s,\dot\gamma_s)
	=& g_{\dot\gamma_s}(D_{\beta_t}^{\dot\gamma_s}\dot\gamma_s,\dot\gamma_s)
	+C_{\dot\gamma_s}(D_{\beta_t}^{\dot\gamma_s}\dot\gamma_s,\dot\gamma_s,\dot\gamma_s)\\ \label{basic}
	=& g_{\dot\gamma_s}(D_{\gamma_s}^{\dot\gamma_s} \dot\beta_t,\dot\gamma_s).
	\end{align} 
	So, the second term on the right-hand side of \eqref{conffirstvar} equals
	\begin{multline}\label{conffirstvar2}
	\int_a^b g_{\dot\gamma_s}(D^{\dot\gamma_s}_{\gamma_s}\dot\beta_t,\lambda(\dot\gamma_s)\dot\gamma_s) dt\\=\int_a^b \frac{d}{dt}g_{\dot\gamma_s}(\dot\beta_t , \lambda(\dot\gamma_s)\dot\gamma_s)dt 
	- \int_a^bg_{\dot\gamma_s}\Big(\dot\beta_t , D_{\gamma_s}^{\dot\gamma_s}(\lambda(\dot\gamma_s)\dot\gamma_s)\Big)dt,
	\end{multline}
	 where we have used again almost $g$-compatibility of the Chern connection and the fact that $C_{\dot\gamma_s}(D_{\gamma_s}^{\dot\gamma_s}\dot\gamma_s,\dot\beta_t,\lambda(\dot\gamma_s)\dot\gamma_s)=0$ by Prop. \ref{Cartanprop}. 
Computing the last terms of \eqref{conffirstvar2} 
in $s=0$,  substituting in \eqref{conffirstvar},  and taking into account that $g_{\dot\gamma_s}(\dot\gamma_s,\dot\gamma_s)=L(\dot\gamma_s)=0$ when $s=0$,  we get (\ref{firstvariation2}).
\end{proof}
 Now observe that given an $A$-admissible curve $\gamma:[a,b]\rightarrow M$, and an arbitrary smooth vector field $W$ along $\gamma$, there always exists a (non-unique) $A$-admissible variation $\Lambda$ of $\gamma$ with $W$ as variational vector field. In fact, it is well-known that we can choose a 
variation $\Lambda:[a,b]\times (-\varepsilon,\varepsilon)\rightarrow M$ of $\gamma$ having $W$ as a variation vector field. As $\Lambda$ is at least $C^1$, being $A$ an open subset  and  $[a,b]$ compact, we can choose a smaller $\varepsilon$ if necessary in such a way that $\Lambda$ is $A$-admissible. Taking into account this fact, 
the last proposition allows us to obtain the geodesic equation for $\lambda L$.
\begin{prop}
The  lightlike  geodesics of the pseudo-Finsler manifold $(M,\lambda L)$ are the smooth lightlike curves which satisfy 
\begin{equation}\label{lightgeo}
  D_\gamma^{\dot\gamma}(\lambda(\dot\gamma)\dot\gamma) = 0.
\end{equation}
\end{prop}
\begin{proof}
 Along the proof, we consider curves with fixed endpoints and smooth $A$-admissible variations.  It is well-known that geodesics are critical points of the energy functional (see for example
 \cite[Cor. 3.7]{JavSoa15}).  Observe that, since  we consider smooth variations, we do not need the
 injectivity of Legendre transform. From Prop. \ref{firstvar2}, we can prove that the critical points of $E_\lambda$  are given by \eqref{lightgeo} analogously to the proof of \cite[Cor. 3.7]{JavSoa15}.
\end{proof}
\begin{rem}\label{reparame}
 Let us observe first that given a curve $\gamma:[a,b]\rightarrow M$ and  a reparametrization $\tilde{\gamma}=\gamma\circ\varphi$ with $\varphi:[\tilde{a},\tilde{b}]\rightarrow [a,b]$  and $\dot\varphi > 0$ on $[\tilde a,\tilde b]$, if $V$ and $W$ are vector fields along $\gamma$ and $\tilde{V}$ and $\tilde{W}$ are the vector fields along $\tilde{\gamma}$ defined as $\tilde{V}(\mu)=V(\varphi(\mu))$ and $\tilde{W}(\mu)=W(\varphi(\mu))$ for any $\mu\in [\tilde{a},\tilde{b}]$, then
\begin{equation*}
D_{\tilde{\gamma}}^{\tilde{V}}\tilde{W}(\mu)=\dot\varphi(\mu)D_\gamma^VW(\varphi(\mu)),
\end{equation*}
for any $\mu\in[\tilde{a},\tilde{b}]$, and $D_\gamma^VW=D_\gamma^{\phi  V}W$ for any function  $\phi :[a,b]\rightarrow (0,+\infty)$ (recall \S \ref{chernsection}).
\end{rem}
\begin{prop}\label{georepara}
 If  the  curve $\gamma:[a,b]\rightarrow M$ is a lightlike geodesic of $(M,\lambda L)$, then $\tilde \gamma=\gamma\circ \varphi$ is a lightlike geodesic of $(M,L)$, where $\varphi:[\tilde{a},\tilde{b}]\rightarrow [a,b]$ is a
solution of the differential equation
\begin{equation}\label{repar}
 \dot\varphi (\mu) = \lambda(\dot\gamma (\varphi (\mu))).
\end{equation}
 \end{prop}
 \begin{proof}
  Observe that  as $\dot{\tilde{\gamma}}(\mu)=\dot\varphi(\mu)\dot\gamma(\varphi(\mu))$,  by applying Remark \ref{reparame}, one gets
\begin{multline*}
D_{\tilde \gamma}^{\dot{\tilde\gamma}}\dot{\tilde\gamma} \left(\mu\right)  = 
 D_{\tilde \gamma}^{ \dot\varphi (\dot{\gamma}\circ\varphi)}\left( \dot\varphi (\dot{\gamma}\circ\varphi)\right)  \left(\mu\right)  \\= 
 D_{\tilde \gamma}^{\dot{\gamma}\circ\varphi}\left( \dot\varphi(\dot{\gamma}\circ\varphi)\right)  \left(\mu\right)  = 
 \dot\varphi(\mu) D_{ \gamma}^{\dot{\gamma}}
 \left(  (\dot\varphi\circ\varphi^{-1}) \dot{\gamma}  \right)  \left(\varphi\left(\mu\right)\right), 
\end{multline*}
 so the conclusion follows from \eqref{lightgeo}. 
 \end{proof}
\subsection{Second variation of the energy}
Our next goal is to study the behavior of conjugate and focal points of lightlike geodesics 
under conformal transformations. We start by computing the second variation of the energy functional $E_\lambda$.

\begin{prop}\label{secondvar2}
Let $\gamma:[a,b]\rightarrow M$ be a lightlike geodesic of $(M,\lambda L)$ and 
consider an $A$-admissible smooth variation $\Lambda$. Then, with the above 
notation,
\begin{multline}\label{confsecondvar}
E_{\lambda}''(0)=\frac{d^2}{ds^2}E_{\lambda}(\gamma_s)\left|_{s=0}\right.\\
= \int_a^b \lambda(\dot\gamma)
\Big(-g_{\dot\gamma}\left(R_{\dot\gamma}\left(\dot\gamma,W\right)W,\dot\gamma\right)+
g_{\dot\gamma}\left(W',W'\right)\Big)dt \\
 + 2\int_a^b g_{\dot\gamma}\left(W',\dot\gamma \right) (g_{\dot\gamma}\left(W,(\nabla^h\lambda)_{\dot\gamma}\right)+g_{\dot\gamma}\left(W',(\nabla^{\nu}\lambda)_{\dot\gamma}\right)) dt
\\+ \left[\lambda(\dot\gamma)g_{\dot\gamma}(D^{\dot\gamma_s}_{\beta_t}\dot\beta_t|_{s=0}, \dot\gamma)\right]_a^b,
\end{multline}
where $D^{\dot\gamma}_{\beta_t}\dot\beta_t|_{s=0}$ is the transverse acceleration 
vector field of the variation, $R$ is the  Chern curvature of $L$ defined in \eqref{RvRV}  and $\, '$ denotes the covariant derivative, namely, $W'=D^{\dot\gamma}_\gamma W$. 
\end{prop}
\begin{proof}
Using   \eqref{conffirstvar} and taking into account that $\gamma_s$ is lightlike for $s=0$,     we get
\begin{multline*}
\frac{d^2}{ds^2}E_{\lambda}(\gamma_s)|_{s=0}=
\int_a^b\left. \frac{d}{ds}\right|_{s=0}\lambda(\dot\gamma_s) \left.\frac{d}{ds}\right|_{s=0}g_{\dot\gamma_s} \left( \dot\gamma_s,\dot\gamma_s\right)dt
\\+\frac 12 \int_a^b \left.\lambda(\dot\gamma) \frac{d^2}{ds^2}\right|_{s=0}g_{\dot\gamma_s} \left( \dot\gamma_s,\dot\gamma_s\right)dt.
\end{multline*}
 Moreover, using \eqref{basic} and \eqref{conffirstvar2}, the above expression becomes 
\begin{multline}\label{confsecvar}
\frac{d^2}{ds^2}E_{\lambda}(\gamma_s)|_{s=0}=
\int_a^b\left. \frac{d}{ds}\right|_{s=0}\lambda(\dot\gamma_s) \left.\frac{d}{ds}\right|_{s=0}g_{\dot\gamma_s} \left( \dot\gamma_s,\dot\gamma_s\right)dt
\\+\int_a^b \lambda\left(\dot\gamma\right)\, \left.\frac{d}{ds}\right|_{s=0}\left( g_{\dot\gamma_s}\left(D_{\gamma_s}^{\dot\gamma_s}\dot\beta_t , \dot\gamma_s\right)\right)dt
\end{multline}
and using the same arguments as in the proof of \cite[Prop. 3.2]{JavSoa15}, and taking into account the observations in the proof of \cite[Prop. 3.7]{J19}, we get that the last term above equals
\begin{equation}\label{confsecvar3}
\int_a^b\lambda(\dot\gamma)\left(\left.g_{\dot\gamma_s}\left(D_{\gamma_s}^{\dot\gamma_s}D_{\beta_t}^{\dot\gamma_s} \dot\beta_t
- R_{\dot\gamma}(\dot\gamma_s,\dot\beta_t)\dot\beta_t ,\dot\gamma_s\right)\right|_{s=0} + 
\left.g_{\dot\gamma_s}\left(D_{\gamma_s}^{\dot\gamma_s} \dot\beta_t ,D_{\gamma_s}^{\dot\gamma_s}\dot\beta_t\right)\right|_{s=0} \right)dt.
\end{equation}
Observe that, as $\gamma$ is not a geodesic of $(M,L)$, but a pregeodesic satisfying \eqref{lightgeo}, in order to get 
\[g_{\dot\gamma}(R_{\dot\gamma}(\dot\gamma,W)W,\dot\gamma)=g_{\dot\gamma}(  R^\gamma(\dot\gamma,W)W,\dot\gamma),\]
as in the proof of \cite[Prop. 3.8]{J19}, one must use that,  by  \eqref{lightgeo}, $D_{\gamma}^{\dot\gamma}\dot\gamma=-\frac{d}{dt}(\lambda(\dot\gamma))\dot\gamma$, and  for any $v\in A$ and $u,w\in T_{\pi(v)}M$, $P_v(u,w,v)=0$ by homogeneity (see \cite[Eq. (1)]{J14}).
 Finally, to compute the first term in \eqref{confsecvar},  we use   \eqref{derivf}  to obtain
\begin{multline}\label{derlambda}
\frac{d}{ds}\left(\lambda(\dot\gamma_s)\right)=\nabla_{\dot \beta_t}\lambda (\dot\gamma_s)+
(\partial^\nu\lambda)_{\dot\gamma_s}(D^{\dot\gamma_s}_{\beta_t}\dot\gamma_s)
=g_{\dot\gamma_s}((\nabla^h\lambda)_{\dot\gamma_s},\dot\beta_t)+g_{\dot\gamma_s}((\nabla^\nu\lambda)_{\dot\gamma_s},D^{\dot\gamma_s}_{\gamma_s}\dot\beta_t),
\end{multline}
where we have used the definition of the vertical and horizontal gradients in \eqref{nablanu} and \eqref{nablah}  and then \eqref{commut}. 

 Substituting 
\eqref{basic}, \eqref{confsecvar3} and  \eqref{derlambda}   in \eqref{confsecvar}, putting $s=0$,  
and using that \[\lambda(\dot\gamma)g_{\dot\gamma}(\left.D_{\gamma_s}^{\dot\gamma_s}D_{\beta_t}^{\dot\gamma_s} \dot\beta_t\right|_{s=0},
\dot\gamma)=\frac{d}{dt}(g_{\dot\gamma}(\left.D_{\beta_t}^{\dot\gamma_s} \dot\beta_t\right|_{s=0},
\lambda(\dot\gamma)\dot\gamma))\]
 for $s=0$ (observe that $\gamma_0=\gamma$ satisfies \eqref{lightgeo}), we get \eqref{confsecondvar}.
\end{proof}

\subsection{$(P,Q)$-Jacobi fields and $P$-focal points}

Recalling  now the notation of \cite[\S 3.2]{JavSoa15}, 
consider the space of curves \[C_L(P,Q)\subset C_L(M,[a,b])\] joining two submanifolds $P$ and $Q$ of $M$, namely,
\[C_L(P,Q):=\{\gamma\in C_L(M,[a,b]): \gamma(a)\in P,\gamma(b)\in Q\}.\]
 When we consider a  smooth $(P,Q)$-variation of $\gamma\in C_L(P,Q)$ 
by curves in $C_L(P,Q)$, the variational vector field is tangent to $P$ and 
$Q$ at the endpoints. Indeed, we define 
 \[T_\gamma C_L(P,Q)=\{W\in T_{\gamma}C_L(M,[a,b]): W(a)\in 
T_{\gamma(a)}P,W(b)\in T_{\gamma(b)}Q\}.\]


 We  denote the tangent bundle of $P$ as $TP$ and define the {\it normal bundle} $TP^\perp$ of  $P$ as the  set of  vectors $v\in A$ such that  $\pi(v)\in P$ and $g_v(v,w)=0$ for every $w\in T_{\pi(v)}P$. We denote  by $\mathcal{F}(P)$ the space of smooth real functions on $P$,  by  $\mathfrak{X}(P)$ the $\mathcal{F}(P)$-module  space of 
smooth sections of the fiber bundle $TP$ over $P$ and by $\mathfrak{X}(P)^\perp$
the space of  smooth sections of $\pi:TP^\perp\rightarrow P_0$,  where $P_0=\pi(TP^\perp)$ (observe that as $L$ is defined in a conic open subset, the intersection $T_pP\cap TP^\perp$ can be empty for some $p\in P$).  
 Given $N\in \mathfrak{X}(P)^\perp$, we denote by $\mathfrak{X}(P)^\perp_N$ 
the subset of smooth sections $W$ of  $\pi:i^*(TM)\rightarrow P$ (where $i^*(TM)$  is the pull-back of $TM$ along the inclusion $i:P\rightarrow M$) 
such that, for every $p\in P$, $W_p$ is $g_{N_p}$-orthogonal to $T_pP$.
\begin{defi}\label{secondfundamental}
Fix $N\in \mathfrak{X}(P)^\perp$ and suppose that $g_{N_p}|_{T_pP\times T_pP}$ is nondegenerate for every $p\in P$. Then the {\it second fundamental form} of $P$ in the direction of $N$ 
 is the map $S^P_N:\mathfrak{X}(P)\times \mathfrak{X}(P)\rightarrow  
\mathfrak{X}(P)^\perp_N$ given by $S_N^P(U,W)={\rm nor}_N \nabla^N_UW$.  Moreover, we define the {\it  normal second fundamental form }  ${\mathcal S}^P_N:\mathfrak{X}(P)\times \mathfrak{X}(P)\rightarrow  
\mathfrak{X}(P)$ as ${\mathcal S}^P_N(U)={\rm tan}_N\nabla^V_UN$.   Here, ${\rm tan}_N$ and ${\rm nor}_N$ compute the tangent and the normal part to $P$, respectively,  using $g_N$. 
\end{defi} 
Observe that the notation for the normal second fundamental form is different from \cite{JavSoa15}. 
When $\gamma\in C_L(P,Q)=C_{\lambda L}(P,Q)$ is a geodesic of $(M,\lambda L)$ which is $g_{\dot\gamma}$-orthogonal to $P$ and $Q$ at the endpoints  and such that $g_{\dot\gamma(a)}|_{P\times P}$ and $g_{\dot\gamma(b)}|_{Q\times Q}$ are nondegenerate, Prop.
\ref{secondvar2} allows us to compute the index form of $\gamma$ as
\begin{multline}\label{formaindice}
I^{\gamma,\lambda}_{P,Q}(V,W)=\int_a^b \lambda(\dot\gamma)\left(-g_{\dot\gamma}(R_{\dot\gamma}(\dot\gamma,V)W,\dot\gamma)+g_{\dot\gamma}(V',W')\right) dt\\
+\int_a^b \left(g_{\dot\gamma}(V',\dot\gamma) g_{\dot\gamma}(W,(\nabla^{\h}\lambda)_{\dot\gamma})+g_{\dot\gamma}(W',\dot\gamma) g_{\dot\gamma}(V,(\nabla^{\h}\lambda)_{\dot\gamma})\right)  dt \\
 +\int_a^b \left(g_{\dot\gamma}(V',\dot\gamma) g_{\dot\gamma}(W',(\nabla^{\nu}\lambda)_{\dot\gamma})+g_{\dot\gamma}(W',\dot\gamma) g_{\dot\gamma}(V',(\nabla^{\nu}\lambda)_{\dot\gamma})\right)  dt  \\
+\lambda(\dot\gamma(b))g_{\dot\gamma(b)}(S^Q_{\dot\gamma(b)}(V,W),\dot\gamma(b))
-\lambda(\dot\gamma(a))g_{\dot\gamma(a)}(S^P_{\dot\gamma(a)}(V,W),\dot\gamma(a)),
\end{multline}
where $V,W\in T_\gamma C_L(P,Q)$,  and $S^P$ and $S^Q$  are  the fundamental forms of $P$ and $Q$ computed with $L$. This comes easily from Prop. \ref{secondvar2}, the equality $E''_\lambda(0)=I^{\gamma,\lambda}_{P,Q}(W,W)$ and the definition of second fundamental form,  taking into account that $g_{\dot\gamma}(R_{\dot\gamma}(\dot\gamma,V)W,\dot\gamma)$ is symmetric in $V$ and $W$, which follows from \cite[Prop. 3.1]{J19}.  

 Recall that given a geodesic $\tilde \gamma:[\tilde a,\tilde b]\rightarrow M$ of a pseudo-Finsler manifold $(M,L)$, we say that a vector field $\tilde J$ along $\tilde \gamma$ is a {\it Jacobi field} if
\begin{equation}\label{jaco}
D_{\tilde{\gamma}}^{\dot{\tilde{\gamma}}}
D_{\tilde{\gamma}}^{\dot{\tilde{\gamma}}}
\tilde{J}= R_{\dot{\tilde{\gamma}}}(\dot {\tilde{\gamma}},\tilde{J})\dot{\tilde{\gamma}},
\end{equation}
where $D_{\tilde \gamma}$ is the covariant derivative associated with the Chern connection and $R$ its  curvature tensor
 (see \cite[Prop. 2.11 and Lemma 3.5]{J19}), 
and we say that it is a {\it $(P,Q)$-Jacobi field} if $\tilde J(\tilde a)$ and $\tilde J(\tilde b)$ are tangent to $P$ and $Q$, respectively, and
\begin{equation}\label{condini}
{\rm tan}_{\dot{\tilde{\gamma}}(a)} (D_{\tilde{\gamma}}^{\dot{\tilde{\gamma}}(a)}\tilde{J}(\tilde{a}))={\mathcal S}^P_{\dot{\tilde{\gamma}}(a)}(\tilde{J}(\tilde{a})),\quad {\rm tan}_{\dot{\tilde{\gamma}}(b)} (D_{\tilde{\gamma}}^{\dot{\tilde{\gamma}}}\tilde{J}(\tilde{b}))={\mathcal S}^Q_{\dot{\tilde{\gamma}}(b)}(\tilde{J}(\tilde{b})),
\end{equation}
where ${\mathcal S}^P$, ${\mathcal S}^Q$ are  the normal second fundamental forms of $P$, $Q$, respectively.
Moreover, we say that $\tilde J$ is a {\it $P$-Jacobi field} if it only satisfies the first identity of \eqref{condini}. Finally, we say that $t_0\in (\tilde a,\tilde b]$ is a {\it $P$-focal point} of $\tilde \gamma$ if there exists a $P$-Jacobi field $\tilde J$ along $\tilde \gamma$ such that  $\tilde J(t_0)=0$. The existence of a $P$-focal point $t_0\in(\tilde a,\tilde b]$ is equivalent to the existence of a $(P,\tilde\gamma(t_0))$-Jacobi field. 

Our next goal is to show that the $P$-focal points of $\gamma$ are preserved with multiplicity in the curve $\tilde{\gamma} =\gamma\circ\varphi$ obtained in Prop. \ref{georepara} (recall the notation of Remark \ref{reparame}).
\begin{lemma}\label{PQJacobi}
Let $\gamma:[a,b]\rightarrow M$ be a  lightlike  geodesic of $(M,\lambda L)$ and $P$ and $Q$ two submanifolds which are orthogonal to $\gamma$ and non-degenerate at $\gamma(a)$ and $\gamma(b)$  with respect to the metrics $g_{\dot\gamma(a)}$ and  $g_{\dot\gamma(b)}$, respectively.   Let $\tilde \gamma$ be the reparametrization of $\gamma$ given in Prop. \ref{georepara} and  assume that $\tilde{J}$ is a $(P,Q)$-Jacobi field of $\tilde{\gamma}$ with the metric $ L$ and $J$ satisfies that $J(\varphi(\mu))=\tilde{J}(\mu)$ for every $\mu\in[\tilde{a},\tilde{b}]$. Then there exists a function $h:[a,b]\rightarrow \R$  with $h(a)=h(b)=0$  such that $\hat{J}(t)=J(t)+h(t) \dot\gamma(t)$, $t\in [a,b]$ is a $(P,Q)$-Jacobi field of $\gamma$ with respect to the metric $ \lambda  L$.
\end{lemma}
\begin{proof}
 Since  $\tilde{J}$ is a Jacobi field of $\tilde{\gamma}$, it holds \eqref{jaco}.
Now observe that $D_{\tilde{\gamma}}^{\dot{\tilde{\gamma}}}\tilde{J}(\mu)= \lambda(\dot\gamma) D^{\dot\gamma}_\gamma J(\varphi(\mu))$, where $\lambda(\dot\gamma)=\dot\varphi$ (recall Rem. \ref{reparame} and Prop. \ref{georepara}), 
\[( R_{\dot{\tilde{\gamma}}}(\dot {\tilde{\gamma}},\tilde{J})\dot{\tilde{\gamma}})(\mu)
=  \lambda(\dot\gamma)^2  (R_{\dot\gamma}(\dot\gamma,J)\dot\gamma) (\varphi(\mu))\]
 and  recall \eqref{repar},  therefore,  using Rem. \ref{reparame} again,  \eqref{jaco} can be rewritten as
\begin{equation}\label{Jacobieq}
(\lambda(\dot\gamma) J')'=\lambda(\dot\gamma) R_{\dot\gamma}(\dot\gamma,J)\dot\gamma,
\end{equation}
and  recalling  \cite[Rem. 3.6]{JavSoa15}, \eqref{condini} becomes
\begin{equation}\label{valoresinicia}
{\rm tan}_{\dot\gamma(a)} J'(a)={\mathcal S}^P_{\dot\gamma(a)}(J(a)),\quad {\rm tan}_{\dot\gamma(b)} J'(b)={\mathcal S}^Q_{\dot\gamma(b)}(J(b))
\end{equation}
(recall that $\, '$ means to apply $D^{\dot\gamma}_\gamma$). By \cite[Prop. 3.11]{JavSoa15}, we know that $(P,Q)$-Jacobi fields are the vector fields in the kernel of the index form. Reasoning as in \cite[Prop. 3.11]{JavSoa15} with the expression \eqref{formaindice}, we get that $V\in C_L(P,Q)$ is a $(P,Q)$-Jacobi field along $\gamma$ if and only if 
\begin{multline}\label{eulag}
\lambda(\dot\gamma) R_{\dot\gamma}(\dot\gamma,V)\dot\gamma-(\lambda(\dot\gamma) V')'+g_{\dot\gamma}(V',\dot\gamma)(\nabla^{\h}\lambda)_{\dot\gamma}-
(g_{\dot\gamma}(V,(\nabla^{\h}\lambda)_{\dot\gamma})\dot\gamma)'\\
 -\left(g_{\dot\gamma}(V',\dot\gamma)(\nabla^{\nu}\lambda)_{\dot\gamma}\right)'-
(g_{\dot\gamma}(V',(\nabla^{\nu}\lambda)_{\dot\gamma})\dot\gamma)'=0
\end{multline}
and
\begin{equation}\label{condiniciales}
 \left[\lambda(\dot\gamma)g_{\dot\gamma}(-{\mathcal S}^Q_{\dot\gamma}(V)+V',W)+g_{\dot\gamma}(g_{\dot\gamma}(V',\dot\gamma)(\nabla^\nu\lambda)_{\dot\gamma},W)\right]^b_a=0,
\end{equation}
   Though $\gamma$ is not a geodesic with respect to $L$, we can get the above equations in a similar way to
the computations of \cite[Prop. 3.11]{JavSoa15} 
using that $\gamma$ is a pregeodesic, namely, 
$D^{\dot\gamma}_\gamma \dot\gamma=\theta \dot\gamma$ for some function 
$\theta:[a,b]\rightarrow\R$, since in this case, the involved Cartan tensor terms are also zero. 
 Now  observe that as $\tilde{J}$ is a $(P,Q)$-Jacobi field for 
$\tilde{\gamma}$, we have that 
$g_{\dot{\tilde{\gamma}}}(D_{\tilde{\gamma}}^{\dot{\tilde{\gamma}}}\tilde{J},\dot{\tilde{\gamma}})=0$ (this follows easily from \cite[Lemma 3.17]{JavSoa15}) and then
\[g_{\dot\gamma(\varphi(\mu))}(J'(\varphi(\mu)),\dot\gamma(\varphi(\mu)))=\frac{1}{\lambda(\dot{\tilde{\gamma}}(\mu))^2}g_{\dot{\tilde{\gamma}}(\mu)}(D_{\tilde{\gamma}}^{\dot{\tilde{\gamma}}}\tilde{J}(\mu),\dot{\tilde{\gamma}}(\mu))=0.\]
Using the last equation, \eqref{lightgeo}, \eqref{Jacobieq}, \eqref{valoresinicia}
and $g_{\dot\gamma}(\dot\gamma,\dot\gamma)=0$, we deduce that if $V(t)=J(t)+h(t)\dot\gamma(t)$ and $h(a)=h(b)=0$, then $V$ satisfies \eqref{eulag} and \eqref{condiniciales} if and only if
\[-(\lambda(\dot\gamma)(h\dot\gamma)')'-(h\dot\lambda\dot\gamma)'= (J(\lambda)(\gamma) \dot\gamma)',\]
where $\dot\lambda(t)=\frac{d}{dt}\lambda(\dot\gamma(t))$  and 
$ J(\lambda)(\gamma) :=g_{\dot\gamma}(J,(\nabla^{\h}\lambda)_{\dot\gamma})+g_{\dot\gamma}(J',(\nabla^{\nu}\lambda)_{\dot\gamma})$. 
  Here, we have also used that $g_{\dot\gamma}(\dot\gamma, (\nabla^\nu\lambda)_{\dot\gamma})=0$ because the $0$-homogeneity of $\lambda$. 
This equation is equivalent to 
\[\frac{d^2}{dt^2}h=-\frac{1}{\lambda(\gamma)}\frac{d}{dt} (J(\lambda)(\gamma)) \dot\gamma + 
\frac{\dot\lambda}{\lambda(\gamma)^2}J(\lambda)(\gamma) \dot\gamma.\]
It is easy to prove that there exists a unique solution $h:[a,b]\rightarrow \R$ of the above differential equation such that $h(a)=h(b)=0$. Then the vector field $\hat{J}=J+h\dot\gamma$ is a $(P,Q)$-Jacobi field along $\gamma$.
\end{proof}
\begin{thm}\label{finalresult}
Assume that $\gamma:[a,b]\rightarrow M$ is a lightlike geodesic of $(M,\lambda L)$, $\tilde{\gamma}=\gamma\circ\varphi$ is the reparametrization as a lightlike geodesic of $(M,L)$ obtained in Prop. \ref{georepara} and $P$, an orthogonal submanifold passing through $\gamma(a)$ and non-degenerate in that point with the metric $g_{\dot\gamma(a)}$. Then $\mu_0\in (\tilde{a},\tilde{b}]$ is a $P$-focal point of $\tilde{\gamma}$ if and only if $\varphi(\mu_0)$ is a $P$-focal point of $\gamma$ with the same multiplicity.
\end{thm}
\begin{proof}
It is a consequence of Lemma \ref{PQJacobi}. Observe that if we choose $Q=\tilde{\gamma}( \mu_0 )$, the above lemma gives a map between $P$-Jacobi fields of $\tilde{\gamma}$ such that $\tilde{J}(\mu_0)=0 $ and $P$-Jacobi fields $\hat J$  of $\gamma$ such that $\hat J (\varphi(\mu_0))=0$. Moreover, this map is injective, because if $\hat{J}=0$, then $\tilde{J}(\mu)=\phi(\mu)\dot{\tilde{\gamma}}(\mu)$ for some smooth function $\phi:[\tilde{a},\tilde{b}]\rightarrow\R$, but from \cite[part $(i)$ of  Lemma 3.17]{JavSoa15}, it follows that $\tilde{J}=0$. The injectivity of the map implies that ${\rm mul}_{\tilde{\gamma}}(\mu_0)\leq {\rm mul}_\gamma(\varphi(\mu_0))$, namely, the multiplicity of $\mu_0$ as a $P$-focal point of $\tilde{\gamma}$ is less or equal to the multiplicity of $\varphi(\mu_0)$ as a $P$-focal point of $\gamma$. Using Lemma \ref{PQJacobi} with the conformal change $1/\lambda$ and the metric $\lambda L$ we get the other inequality concluding that ${\rm mul}_{\tilde{\gamma}}(\mu_0)= {\rm mul}_{\gamma}(\varphi(\mu_0))$ as required.
\end{proof}

\end{document}